\documentclass[12pt]{article}%
\usepackage{graphicx}
\usepackage{color}
\usepackage{amsmath}
\usepackage{a4}
\usepackage{amsfonts}
\usepackage{amssymb}
\usepackage[all]{xy}
\providecommand{\U}[1]{\protect \rule{.1in}{.1in}}
\newtheorem{theorem}{Theorem}

\newtheorem{corollary}[theorem]{Corollary}

\newtheorem{example}[theorem]{Example}

\newtheorem{lemma}[theorem]{Lemma}

\newtheorem{proposition}[theorem]{Proposition}

\newenvironment{proof}[1][Proof]{\textbf{#1.} }{\  \rule{0.5em}{0.5em}}

\begin{document}

\title{On surfaces with conical singularities of arbitrary angle}
\author{Charalampos Charitos$\dagger$, Ioannis Papadoperakis$\dagger$
\and and Georgios Tsapogas$\ddagger$\\$\dagger$Agricultural University of Athens \\and $\ddagger$University of the Aegean}
\maketitle

\begin{abstract}
The geometry of closed surfaces equipped with a Euclidean metric with finitely
many conical points of arbitrary angle is studied. The main result is that the
set of closed geodesics is dense in the space of geodesics. \newline%

\end{abstract}

\section{Introduction}

Let $S$ be a closed surface of genus $\geq1$ equipped with a Euclidean metric
with finitely many \textit{conical singularities} (or \textit{conical
points}), denoted by $s_{1},...,s_{n}.$ Every point which is not conical will
be called a \textit{regular point} of $S.$ Denote by $\theta(s_{i})$ the angle
at each $s_{i},$ with $\theta(s_{i})\in(0,+\infty)\setminus \left \{
2\pi \right \}  $ and denote by $C\left(  S\right)  $ the set $\left \{
s_{1},...,s_{n}\right \}  .$ Existence of conical points with angle less than
$2\pi$ results into major differences in the geometry of the surface $S$
compared to the case where all angles are $>2\pi.$ An\ important property
which fails in this class of spaces is that geodesic segments with specified
endpoints and homotopy class are no longer unique and similarly for geodesic
rays and lines in the universal cover $\widetilde{S}$ (see Example
\ref{paradeigma} below). Moreover, extension of geodesics also fails. More
precisely, there exist geodesic segments $\sigma$ in $\widetilde{S}$ not
containing any singularity which cannot be extended to any geodesic segment
$\sigma^{\prime}$ properly containing $\sigma.$ These facts make the study of
the geometry of $\widetilde{S}$ interesting. In fact, the tools for studying
the geometry of Euclidean surfaces with conical points of arbitrary angle are,
in principle, limited to the property that $\widetilde{S}$ is a hyperbolic
space in the sense of Gromov.

In this note we first show that the set of points on the boundary
$\partial \widetilde{S}$ to which there corresponds more than one geodesic ray
is dense in $\partial \widetilde{S}$ (see Theorems \ref{dense_notunique} and
\ref{line_dense} below). Moreover, it is shown that for any boundary point
$\xi \in \partial \widetilde{S}$ there exists a base point $\widetilde{x}_{0}%
\in \widetilde{S}$ such that at least two geodesic rays emanating from
$\widetilde{x}_{0}$ correspond to $\xi.$ We then show that the images of all
geodesic rays corresponding to a boundary point $\xi \in \partial \widetilde{S}$
are contained in a convex subset of $\widetilde{S}$ whose boundary is geodesic
consisting of two geodesic rays. Thus, in the class of geodesic rays
corresponding to each point $\xi \in \partial \widetilde{S}$ there are associated
two distinct outermost (left and right) geodesic rays. Similarly for geodesic lines.

We show that the set of closed geodesics is dense in the space of all
geodesics $GS$ in the following sense: for each pair of distinct points
$\xi,\eta \in \partial \widetilde{S}$ and each outermost geodesic line $\gamma$
joining them, there exists a sequence of geodesics $\left \{  c_{n}\right \}  $
in $\widetilde{S},$ with the projection of every $c_{n}$ to $S$ being a closed
geodesic, such that $\left \{  c_{n}\right \}  $ converges in the usual uniform
sense on compact sets to $\gamma.$

\section{Preliminaries}

Let $\widetilde{S}$ be the universal covering of $S$ and let $p:\widetilde
{S}\rightarrow S$ be the universal covering projection. Obviously, the
universal covering $\widetilde{S}$ is homeomorphic to $\mathbb{R}^{2}$ and by
requiring $p$ to be a local isometric map we may lift $d$ to a metric
$\widetilde{d}$ on $\widetilde{S}$ so that $(\widetilde{S},$ $\widetilde{d})$
becomes an $e.s.c.s.$ Clearly, $\pi_{1}\left(  S\right)  $ is a discrete group
of isometries of $\widetilde{S}$ acting freely on $\widetilde{S}$ so that
$S=\widetilde{S}/\pi_{1}\left(  S\right)  .$

Due to the existence of conical points with angle $<2\pi,$ a geodesic $\gamma$
in $S$, usually defined to be a local isometric map, may have homotopically
trivial self intersections, that is,

\begin{center}
$\exists t_{1},t_{2}\in \mathbb{R}$ with $\gamma \left(  t_{1}\right)
=\gamma \left(  t_{2}\right)  $ such that the loop $\gamma|_{\left[
t_{1},t_{2}\right]  }$ is contractible.
\end{center}

\noindent Clearly, any lift $\widetilde{\gamma}$ to the universal cover
$\widetilde{S}$ of $S$ of a local geodesic $\gamma$ with homotopically trivial
self intersections is not a global isometric map. In view of this and Lemma
\ref{positive_distance} below, we restrict our attention to geodesics and
geodesic segments which do not have homotopically trivial self intersections.

Let $GS$ be the space of all local isometric maps $\gamma:\mathbb{R}%
\rightarrow S$ so that its lift to the universal cover $\widetilde{S}$ is a
(global) geodesic. The image of such a $\gamma$ will be referred to as a
geodesic in $S.$ Similarly we define the notion of a geodesic segment, that
is, a local isometric map whose domain is a closed interval which lifts to a
geodesic segment in $\widetilde{S}.$ The geodesic flow is defined by the map

\begin{center}
$\Phi:\mathbb{R}\times GS\rightarrow GS$
\end{center}

\noindent where the action of $\mathbb{R}$ is given by right translation, i.e.
for each $t\in \mathbb{R}$ and $\gamma \in GS,$ $\Phi(t,\gamma)=t\cdot \gamma,$
where $t\cdot \gamma:\mathbb{R}\rightarrow S$ is the geodesic defined by
$t\cdot \gamma(s)=\gamma(t+s),$ $s\in \mathbb{R}.$

The group $\pi_{1}\left(  S\right)  $ with the word metric is
\textit{hyperbolic in the sense of Gromov. }On the other hand, $\pi_{1}\left(
S\right)  $ acts co-compactly on $\widetilde{S};$ this implies that
$\widetilde{S}$ is itself a hyperbolic space in the sense of Gromov (see for
example \cite[Ch.4 Th. 4.1]{[CDP]}) which is complete and locally compact.
Hence, $\widetilde{S}$ is a proper space i.e. each closed ball in
$\widetilde{S}$ is compact (see \cite{[GLP]} Th. 1.10). Therefore, the visual
boundary $\partial \widetilde{S}$ of $\widetilde{S}$ is defined by means of
geodesic rays and is homeomorphic to $\mathbb{S}^{1}$ (see \cite{[CDP]}, p.19).

The following properties contain information concerning the images of
geodesics with respect to the conical points.

\begin{lemma}
\label{positive_distance}There exists a positive real number $C$ such that for
any geodesic $\gamma$ in $\widetilde{S}$
\[
d\left(  \gamma \left(  t\right)  ,\widetilde{s_{i}}\right)  \geq
C\mathrm{\ for\ all\  \ }t\in \mathbb{R}\mathrm{\  \  \mathrm{and\  \ for\  \ all\ }%
}s_{i}\in C\left(  S\right)  \mathrm{\ with\ }\theta \left(  s_{i}\right)
\in \left(  0,2\pi \right)  .
\]
where $\widetilde{s_{i}}$ denotes a pre-image of $s_{i}.$
\end{lemma}

The proof of this Lemma is given in \cite{[CPT]}. By considering, if
necessary, a constant $C^{\prime}$ smaller than $C,$ the above Lemma holds for
geodesics in $S.$

\begin{lemma}
\label{dyoBig}If two geodesic segments $\sigma_{1},\sigma_{2}$ in
$\widetilde{S}$ intersect at two points $x,y$ such that $x,y$ are isolated in
$\sigma_{1}\cap \sigma_{2}$ then both $x,y$ are conical points with angle
$>2\pi.$ If $\sigma_{1}\cap \sigma_{2}$ is a closed segment then its endpoints
are conical points with angle $>2\pi.$\newline Similarly for homotopic with
endpoints fixed geodesic segments in $S.$
\end{lemma}

\begin{proof}
Let $\sigma_{1}=\left[  w_{1},z_{1}\right]  ,$ $\sigma_{2}=\left[  w_{2}%
,z_{2}\right]  $ be two geodesic segments in $\widetilde{S}$ intersecting at
two points $x,y$ which are isolated in $\sigma_{1}\cap \sigma_{2}.$ Clearly,
$\sigma_{1}|_{\left[  x,y\right]  }\cup \sigma_{2}|_{\left[  y,z_{2}\right]  }$
realizes the distance from $x$ to $z_{2}.$ Therefore, the angle formed by
$\sigma_{1}|_{\left[  x,y\right]  }$, $\sigma_{2}|_{\left[  y,z_{2}\right]  }$
at $y$ is at least $\pi.$ Similarly, the angle formed by $\sigma_{2}|_{\left[
x,y\right]  }$, $\sigma_{1}|_{\left[  y,z_{1}\right]  }$ at $y$ is at least
$\pi,$ hence $\theta \left(  y\right)  >2\pi.$
\end{proof}

Since $\widetilde{S}$ is a hyperbolic space in the sense of Gromov, the
isometries of $\widetilde{S}$ are classified as elliptic, parabolic and
hyperbolic \cite{Gr}. On the other hand, $\pi_{1}\left(  S\right)  $ is a
hyperbolic group, thus $\pi_{1}\left(  S\right)  $ does not contain parabolic
elements with respect to its action on its Cayley graph (see Th. 3.4 in
\cite{[CDP]}). From this, it follows that all elements of $\pi_{1}\left(
S\right)  $ are hyperbolic isometries of $\widetilde{S}.$ Therefore, for each
$\varphi \in \pi_{1}\left(  S\right)  $ and each $x\in \widetilde{S}$ the
sequence $\varphi^{n}(x)$ (resp. $\varphi^{-n}(x))$ has a limit point
$\varphi(+\infty)$ (resp. $\varphi(-\infty))$ when $n\rightarrow+\infty$ and
$\varphi(+\infty)\neq \varphi(-\infty).$ The point $\varphi(+\infty)$ is called
\textit{attractive} and the point $\varphi(-\infty)$ \textit{repulsive} point
of $\varphi.$

The following important property for hyperbolic spaces (see Proposition 2.1 in
\cite{[CDP]}) holds for $\partial \widetilde{S}.$

\begin{proposition}
\label{existence rays, lnes}For every pair of points $x\in$ $\widetilde{S},$
$\xi \in \partial \widetilde{S}$ (resp. $\eta,\xi \in \partial \widetilde{S})$ there
is a geodesic ray $r:[0,\infty)\rightarrow \widetilde{S}\cup \partial
\widetilde{S}$ (resp. a geodesic line $\gamma:(-\infty,\infty)\rightarrow
\widetilde{S}\cup \partial \widetilde{S})$ such that, $r(0)=x,$ $r(\infty)=\xi$
(resp. $\gamma(-\infty)=\eta,$ $\gamma(\infty)=\xi).$
\end{proposition}

Remark that uniqueness does not hold in the above proposition. In fact, we
have the following straightforward corollary to Lemma \ref{dyoBig}.

\begin{corollary}
\label{Cor2Big}If a geodesic segment intersects a geodesic ray at two 
isolated points as in Lemma \ref{dyoBig},
then there exist two distinct geodesic rays defining the same point at
infinity.\newline Similarly for geodesic lines.
\end{corollary}

Thus, for each pair of points $x\in$ $\widetilde{S},$ $\xi \in \partial
\widetilde{S}$ there corresponds a class of geodesic rays $r$ with $r(0)=x,$
$r(\infty)=\xi,$ the cardinality of which varies from a singleton to
uncountable (see discussion following Example \ref{paradeigma} below). It is
well known that in hyperbolic metric spaces the stability property of
quasi-geodesic rays (and lines) holds in the sense of bounded Hausdorff
distance (see, for example, \cite[Chapter I, \S \ 6]{[Coo]}). It follows that
any two geodesic rays $r_{1},r_{2}$ in the same class are asymptotic, that is,
(see \cite[Ch. 3, Thm. 3.1]{[CDP]}) there exists a constant $A>0$ which
depends only on the hyperbolicity constant of $\widetilde{S}$ such that
\begin{equation}
\forall t\in \left[  0,+\infty \right)  ,d\left(  r_{1}\left(  t\right)
,r_{2}\left(  t\right)  \right)  <A. \label{asymptoric_constant}%
\end{equation}
Similarly for geodesic lines. For a point $\xi \in \partial \widetilde{S}$ (and
having fixed a base point in $\widetilde{S})$ we write $r\in \xi$ to indicate
that the geodesic ray $r$ belongs to the class of rays corresponding to $\xi,$
that is, $r(\infty)=\xi.$ We also say that $\xi$ is the positive point of $r.$

Similarly, for a pair $\left(  \eta,\xi \right)  $ of points in $\partial
\widetilde{S}$ with $\eta \neq \xi$ we write $\gamma \in \left(  \eta,\xi \right)
$ to indicate that the geodesic line $\gamma$ belongs to the class of lines
with the property $\gamma(-\infty)=\eta$ and $\gamma(\infty)=\xi.$ We say that
$\xi$ is the positive point of $\gamma$ and $\eta$ the negative.

By writing that the sequence $\left \{  \xi_{n}\right \}  \subset \partial
\widetilde{S}$ converges to $\xi$ in the visual metric, notation $\xi
_{n}\rightarrow \xi,$ we mean that there exist geodesic rays $r_{n}\in \xi_{n}$
and $r\in \xi$ such that the sequence $\left \{  r_{n}\right \}  $ converges in
the usual uniform sense on compact sets to $r.$

The following example demonstrates a simple case where lifts of distinct
closed geodesic (as well as non-closed geodesics) have the same negative and
positive points in $\partial \widetilde{S}.$

\begin{figure}[ptb]
\begin{center}
\includegraphics[scale=0.8]
{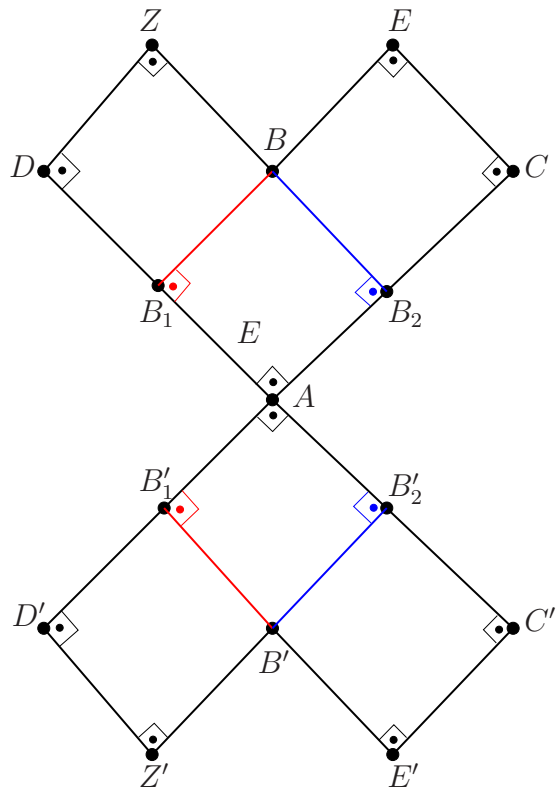}
\end{center}
\par
\begin{picture}(22,12)
\put(211,162){$A$}
\put(247,194){$B_2$}
\put(247,128){$B_2^{\prime}$}
\put(247,18){$E^{\prime}$}
\put(247,304){$E$}
\put(153,194){$B_1$}
\put(153,130){$B_1^{\prime}$}
\put(153,18){$Z^{\prime}$}
\put(153,304){$Z$}
\put(104,249){$D$}
\put(105,79){$D^{\prime}$}
\put(299,249){$C$}
\put(299,77){$C^{\prime}$}
\put(200,259){$B$}
\put(198,61){$B'$}
\put(190,186){$E$}
\end{picture}
\caption{The surface $\Sigma$ with two conical points of angle $\pi$ and
$3\pi.$}%
\label{example}%
\end{figure}

\begin{example}
\label{paradeigma}Consider the genus $0$ surface $\Sigma$ obtained from the
flat figures $ACEBZDA$ and $AC^{\prime}E^{\prime}B^{\prime}Z^{\prime}%
D^{\prime}A$ by identifying $AB_{2}C$ with $AB_{2}^{\prime}C^{\prime},$
$AB_{1}D$ with $AB_{1}^{\prime}D^{\prime}$ and $EBZ$ with $E^{\prime}%
B^{\prime}Z^{\prime}$ (see Figure \ref{example}). The resulting cylinder
$\Sigma$ has two singular points $A,B$ with angles $\theta \left(  A\right)
=\pi$ and $\theta \left(  B\right)  =3\pi.$ The segments $BB_{1}$ and
$B_{1}^{\prime}B^{\prime}$ give rise to a simple closed geodesic $\sigma$ in
$\Sigma.$ Similarly, the segments $BB_{2}$ and $B_{2}^{\prime}B^{\prime}$ give
rise to a simple closed geodesic $\tau$ in $\Sigma.$ Both $\sigma$ and $\tau$
contain $B$ and their union bounds a convex subset of $\Sigma$ with the same
homotopy type as $\Sigma.$
\end{example}

Since $\Sigma$ has geodesic boundaries, the described example can clearly
occur in surfaces of any genus. Pick a lift $\widetilde{\sigma}$ of $\sigma$
in $\widetilde{\Sigma}.$ Then there is a countable number of points
$\widetilde{B_{i}},$ $i\in \mathbb{Z},$ with the properties $\widetilde{B_{i}%
}\in \operatorname{Im}\widetilde{\sigma}$ and $p\left(  \widetilde{B_{i}%
}\right)  =B.$ Clearly, any lift $\widetilde{\tau}$ of $\tau$ containing
$\widetilde{B_{i_{0}}}$ for some $i_{0}$ must contain $\widetilde{B_{i}}$ for
all $i$ and, moreover, $\widetilde{\tau}\left(  +\infty \right)  =\widetilde
{\sigma}\left(  +\infty \right)  $ and $\widetilde{\tau}\left(  -\infty \right)
=\widetilde{\sigma}\left(  -\infty \right)  .$ Therefore, using $\sigma$ and
$\tau$ we may construct countably many pairwise distinct closed geodesics in
$\Sigma,$ as well as uncountably many non-closed geodesics, whose lifts in
$\widetilde{\Sigma}$ are contained in $\operatorname{Im}$ $\widetilde{\tau
}\cup \operatorname{Im}\widetilde{\sigma}$ and they all share the same positive
(resp. negative) point $\widetilde{\tau}\left(  +\infty \right)  =\widetilde
{\sigma}\left(  +\infty \right)  $ (resp. $\widetilde{\tau}\left(
-\infty \right)  =\widetilde{\sigma}\left(  -\infty \right)  $ )$.$

The \textit{limit set} $\Lambda$ of $\pi_{1}\left(  S\right)  $ is defined to
be $\Lambda=\overline{\pi_{1}\left(  S\right)  x}\cap \partial \widetilde{S},$
where $x$ is an arbitrary point in $\widetilde{S}.$ Since the action of
$\pi_{1}\left(  S\right)  $ on $\widetilde{S}$ is co-compact, it is a well
known fact that $\Lambda=\partial \widetilde{S},$ and hence $\Lambda
=\mathbb{S}^{1}.$ Note that the action of $\pi_{1}\left(  S\right)  $ on
$\widetilde{S}$ can be extended to $\partial \widetilde{S}$ and that the action
of $\pi_{1}\left(  S\right)  $ on $\partial \widetilde{S}\times \partial
\widetilde{S}$ is given by the product action.

Denote by $F_{h}$ the set of points in $\partial \widetilde{S}$ which are fixed
by hyperbolic elements of $\pi_{1}\left(  S\right)  .$ Since $\Lambda
=\partial \widetilde{S},$ the following three results can be derived from
\cite{[Coo]}.

\begin{proposition}
\label{dense1} The set $F_{h}$ is $\pi_{1}\left(  S\right)  -$invariant and
dense in $\partial \widetilde{S}.$
\end{proposition}

\begin{proposition}
\label{dense2} There exists an orbit of $\pi_{1}\left(  S\right)  $ dense in
$\partial \widetilde{S}\times \partial \widetilde{S}.$
\end{proposition}

\begin{proposition}
\label{dense3}The set $\left \{  \left(  \phi(+\infty),\phi(-\infty)\right)
:\phi \in \pi_{1}\left(  S\right)  \right \}  $ is dense in $\partial
\widetilde{S}\times \partial \widetilde{S}.$
\end{proposition}

\section{Density in $\partial \widetilde{S}.$}

In this section we first show that the set of points in $\partial \widetilde
{S}$ for which the class of the corresponding geodesic rays is not a
singleton, forms a dense subset of $\partial \widetilde{S}.$ We fix throughout
a base point $\widetilde{x_{0}}\in \widetilde{S}.$

\begin{proposition}
\label{dense_notunique} The set
\[
Y=\left \{  \xi \in \partial \widetilde{S}\bigm \vert%
\begin{array}
[c]{c}%
\exists \mathrm{\mathrm{\  \ distinct\ }\ geodesic\mathrm{\  \ }rays\mathrm{\  \ }%
}r_{1},r_{2}\mathrm{\mathrm{\ }\ such\mathrm{\  \ that}}\\
r_{1}\left(  0\right)  =\widetilde{x_{0}}=r_{2}\left(  0\right)
\mathrm{\  \ and\  \ }r_{1}\left(  \infty \right)  =\xi=r_{2}\left(
\infty \right)
\end{array}
\right \}
\]
is dense in $\partial \widetilde{S}.$
\end{proposition}

\begin{proof}
As $\partial \widetilde{S}$ is homeomorphic to $\mathbb{S}^{1}$ we will be
talking about intervals in $\partial \widetilde{S}$ and we will mean open
(resp. closed) connected subsets of $\partial \widetilde{S}$ homeomorphic to
open (resp. closed) intervals in $\mathbb{S}^{1}.$ It suffices to show that
for any interval $I\subset \partial \widetilde{S},$ $I\cap Y\neq \emptyset
.$\newline Claim: Let $\xi \notin Y,$ $r_{\xi}$ the (unique) geodesic ray with
$r_{\xi}\left(  0\right)  =\widetilde{x}_{0},$ $r_{\xi}\left(  +\infty \right)
=\xi$ and $r$ an arbitrary geodesic ray with $r\left(  0\right)
=\widetilde{x}_{0}$ and $r\left(  +\infty \right)  \neq \xi.$ Then
$\operatorname{Im}r_{\xi}\cap \operatorname{Im}r$ is either, a geodesic segment
of the form $\left[  \widetilde{x}_{0},\widetilde{x}_{1}\right]  $ for some
$\widetilde{x}_{1}\in \operatorname{Im}r_{\xi}$ or, a singleton namely
$\left \{  \widetilde{x}_{0}\right \}  .$\newline Similarly, if $r$ is an
arbitrary geodesic segment then $\operatorname{Im}r_{\xi}\cap \operatorname{Im}%
r$ is either, a geodesic sub-segment of $r$ or, a singleton or, the empty set.

For the proof of the Claim observe that $\operatorname{Im}r_{\xi}%
\cap \operatorname{Im}r$ is necessarily connected. For, if $x,y$ belong to
distinct connected components of $\operatorname{Im}r_{\xi}\cap
\operatorname{Im}r$ then $r|_{\left[  x,y\right]  }$ does not coincide with
$r_{\xi}|_{\left[  x,y\right]  }.$ Thus, the geodesic ray
\[
r^{\prime}=r_{\xi}|_{\left[  \widetilde{x}_{0},x\right]  }\cup r|_{\left[
x,y\right]  }\cup r_{\xi}|_{\left[  y,+\infty \right]  }%
\]
is distinct from $r_{\xi}$ and, clearly, $r^{\prime}\left(  +\infty \right)
=\xi,$ a contradiction. As both $\operatorname{Im}r_{\xi},\operatorname{Im}r$
are homeomorphic to $\left[  0,+\infty \right)  ,$ the Claim follows. The proof
in the case $r$ is a geodesic segment is similar.

Returning to the proof of the Proposition, suppose, on the contrary, that for
some closed interval $\left[  \eta,\rho \right]  \subset \partial \widetilde{S}$
we have $\left[  \eta,\rho \right]  \cap Y=\emptyset.$ By the Claim,
$\operatorname{Im}r_{\eta}\cap \operatorname{Im}r_{\rho}=\left \{  \widetilde
{x}_{0}\right \}  $ or, $\left[  \widetilde{x}_{0},\widetilde{x}_{1}\right]  .$
We may assume that $\operatorname{Im}r_{\eta}\cap \operatorname{Im}r_{\rho
}=\left \{  \widetilde{x}_{0}\right \}  $ otherwise, replace in the sequel the
union $\operatorname{Im}r_{\eta}\cup \operatorname{Im}r_{\rho}$ by
\[
\operatorname{Im}r_{\eta}\cup \operatorname{Im}r_{\rho}\setminus \left[
\widetilde{x}_{0},\widetilde{x}_{1}\right)  .
\]
Then, inside the compact, convex set $\widetilde{S}\cup \partial \widetilde{S}$
the union%
\[
\operatorname{Im}r_{\eta}\cup \operatorname{Im}r_{\rho}\cup \left[  \eta
,\rho \right]
\]
splits the set $\widetilde{S}\cup \partial \widetilde{S}$ into two convex
(closed) subsets whose common boundary is the union $\operatorname{Im}r_{\eta
}\cup \operatorname{Im}r_{\rho}\cup \left \{  \eta,\rho \right \}  .$ Observe that
convexity follows from the above Claim and the assumption $\left[  \eta
,\rho \right]  \cap Y=\emptyset.$ Denote by $\widetilde{S}\left(  \left[
\eta,\rho \right]  \right)  $ the subset of $\widetilde{S}\cup \partial
\widetilde{S}$ which contains $\left[  \eta,\rho \right]  ,$ Pick and fix a
conical point $\widetilde{s}$ in the interior of $\widetilde{S}\left(  \left[
\eta,\rho \right]  \right)  $ with $\theta \left(  \widetilde{s}\right)  <2\pi.$

For each $\xi \in \left[  \eta,\rho \right]  ,$ consider the (unique, as $\left[
\eta,\rho \right]  \cap Y=\emptyset$) geodesic ray $r_{\xi}$ with $r_{\xi
}\left(  0\right)  =x_{0}$ and $r_{\xi}\left(  +\infty \right)  =\xi.$ By the
Claim, as above, $\operatorname{Im}r_{\xi}\cup \left \{  \xi \right \}  $ splits
$\widetilde{S}\left(  \left[  \eta,\rho \right]  \right)  $ into two closed
convex subsets $\widetilde{S}\left(  \left[  \eta,\xi \right]  \right)  $ and
$\widetilde{S}\left(  \left[  \xi,\rho \right]  \right)  ,$ the former
containing $\eta$ and the latter containing $\rho,$ whose intersection is
$\operatorname{Im}r_{\xi}\cup \left \{  \xi \right \}  .$ Define
\[
I\left(  \eta \right)  =\left \{  \xi \in \left[  \eta,\rho \right]
\bigm \vert \widetilde{s}\notin \widetilde{S}\left(  \left[  \eta,\xi \right]
\right)  \right \}  .
\]
Similarly, define $I\left(  \rho \right)  .$ We will show that $I\left(
\eta \right)  ,I\left(  \rho \right)  $ are closed and disjoint subsets of
$\left[  \eta,\rho \right]  ,$ thus contradicting the connectedness of $\left[
\eta,\rho \right]  .$\newline Let $\xi \in I\left(  \eta \right)  \cap I\left(
\rho \right)  .$ Then, by definition, $\widetilde{s}\in$ $\widetilde{S}\left(
\left[  \eta,\xi \right]  \right)  \cap \widetilde{S}\left(  \left[  \xi
,\rho \right]  \right)  .$ It follows that $\widetilde{s}\in \operatorname{Im}%
r_{\xi}$ which contradicts Lemma \ref{positive_distance}. This shows that
$I\left(  \eta \right)  ,I\left(  \rho \right)  $ are disjoint.

To see that $I\left(  \eta \right)  $ is closed, let $\left \{  \xi_{n}\right \}
\subset I\left(  \eta \right)  $ be a sequence converging to $\xi$ with
$r_{\xi_{n}},r_{\xi}$ the corresponding (unique) geodesic rays with positive
points $\xi_{n},\xi.$ We want to show that $\xi \in I\left(  \eta \right)  .$
Assume, on the contrary, that $\xi \in I\left(  \rho \right)  ,$ i.e
$\widetilde{s}\notin \widetilde{S}\left(  \left[  \xi,\rho \right]  \right)  .$
Let $\left[  x_{\eta},\widetilde{s}\right]  $ (resp. $\left[  x_{\rho
},\widetilde{s}\right]  )$ be geodesic segments realizing the distance, say
$d_{\eta}$ (resp. $d_{\rho}$) of $\widetilde{s}$ form $\operatorname{Im}%
r_{\eta}$ (resp. $\operatorname{Im}r_{\rho}$) for some point $x_{\eta}%
\in \operatorname{Im}r_{\eta}$ (resp. $x_{\rho}\in \operatorname{Im}r_{\rho}$).
Consider the following neighborhood $\mathcal{N}$ of geodesic rays around
$r_{\xi}$ determined by the positive number $C/2$ (cf. Lemma
\ref{positive_distance}) and the compact set $\left[  0,2\left(  d_{\eta
}+d_{\rho}\right)  \right]  :$%
\[
\mathcal{N}=\left \{  r\bigm \vert r\left(  0\right)  =\widetilde{x}%
_{0}\mathrm{\mathrm{\ }and\mathrm{\ }}\forall t\in \left[  0,2\left(  d_{\eta
}+d_{\rho}\right)  \right]  ,d\left(  r\left(  t\right)  ,r_{\xi}\left(
t\right)  \right)  <C/2\right \}  .
\]
Note that if $r\in \mathcal{N},$ then $\operatorname{Im}r|_{\left[  2\left(
d_{\eta}+d_{\rho}\right)  ,+\infty \right)  }$ does not intersect the union of
segments $\left[  x_{\eta},\widetilde{s}\right]  \cup \left[  x_{\rho
},\widetilde{s}\right]  ,$ otherwise, $r$ would not be a (global) geodesic
ray. Clearly, for all $n$ large enough, $r_{\xi_{n}}\in \mathcal{N},$ thus,
$r_{\xi_{n}}\left(  t\right)  \notin \left[  x_{\eta},\widetilde{s}\right]
\cup \left[  x_{\rho},\widetilde{s}\right]  $ for all $t>2\left(  d_{\eta
}+d_{\rho}\right)  ,$ which implies that $\widetilde{s}\notin \widetilde
{S}\left(  \left[  \xi_{n},\rho \right]  \right)  .$ In other words, $\xi
_{n}\in I\left(  \rho \right)  $, a contradiction.
\end{proof}

\noindent We now show the analogous result for geodesic lines. We write
$\partial^{2}\widetilde{S}$ for the product $\partial \widetilde{S}%
\times \partial \widetilde{S}$ with the diagonal excluded.

\begin{proposition}
\label{line_dense}The set
\[
Z=\left \{  \left(  \eta,\xi \right)  \in \partial^{2}\widetilde{S}\biggm \vert%
\begin{array}
[c]{c}%
\exists \mathrm{\mathrm{\ distinct\ }geodesic\mathrm{\ lines\ }}\gamma
_{1},\gamma_{2}\mathrm{\mathrm{\ }such\mathrm{\ that}}\\
\gamma_{1}\left(  -\infty \right)  =\eta=\gamma_{2}\left(  -\infty \right)
\mathrm{\  \ and\  \ }\gamma_{1}\left(  \infty \right)  =\xi=\gamma_{2}\left(
\infty \right)
\end{array}
\right \}
\]
is dense in $\partial^{2}\widetilde{S}.$
\end{proposition}

\begin{proof}
It suffices to show that for arbitrary $\xi_{0}\in \partial \widetilde{S}$ and
any closed interval $\left[  \eta,\rho \right]  \subset \partial \widetilde{S}$
with $\xi_{0}\notin \left[  \eta,\rho \right]  ,$ there exist two distinct
geodesics $\gamma_{1},\gamma_{2}$ with $\gamma_{1}\left(  -\infty \right)
=\xi_{0}=\gamma_{2}\left(  -\infty \right)  $ and $\gamma_{1}\left(
+\infty \right)  =\xi=\gamma_{2}\left(  +\infty \right)  .$ For, if $\forall
\xi \in \left[  \eta,\rho \right]  $ there exist a unique geodesic line
$\gamma_{\xi}$ with $\gamma_{\xi}\in \left(  \xi_{0},\xi \right)  $ we may
repeat the argument in the proof of the previous proposition as follows: pick
and fix a singular point $\widetilde{s}$ with $\theta \left(  \widetilde
{s}\right)  <2\pi$ in the interior of the (convex) set $\widetilde{S}\left(
\left[  \eta,\rho \right]  \right)  $ bounded by the geodesic lines joining the
pairs $(\xi_{0},\eta)$ and $(\xi_{0},\rho).$ For each $\xi \in \left[  \eta
,\rho \right]  ,$ consider the (unique, as $\left[  \eta,\rho \right]  \cap
Z=\emptyset$) geodesic $\gamma_{\xi}$ with $\gamma_{\xi}\left(  -\infty
\right)  =\xi_{0}$ and $\gamma_{\xi}\left(  +\infty \right)  =\xi.$ By the
Claim, as above, $\operatorname{Im}\gamma_{\xi}\cup \left \{  \xi \right \}  $
splits $\widetilde{S}\left(  \left[  \eta,\rho \right]  \right)  $ into two
closed convex subsets $\widetilde{S}\left(  \left[  \eta,\xi \right]  \right)
$ and $\widetilde{S}\left(  \left[  \xi,\rho \right]  \right)  ,$ the former
containing $\eta$ and the latter containing $\rho,$ whose intersection is
$\operatorname{Im}\gamma_{\xi}\cup \left \{  \xi_{0},\xi \right \}  .$ Define
\[
I\left(  \eta \right)  =\left \{  \xi \in \left[  \eta,\rho \right]
\bigm \vert \widetilde{s}\notin \widetilde{S}\left(  \left[  \eta,\xi \right]
\right)  \right \}  .
\]
Similarly, define $I\left(  \rho \right)  .$ Then, as above, $I\left(
\eta \right)  ,I\left(  \rho \right)  $ are closed and disjoint subsets of
$\left[  \eta,\rho \right]  ,$ a contradiction.

\noindent Proof of $I\left(  \eta \right)  $ closed: 
Let $\left \{  \xi_{n}\right \}  \subset I\left(  \eta \right)  $ be a sequence
converging to $\xi$ with $\gamma_{\xi_{n}},\gamma_{\xi}$ the corresponding
(unique) geodesic lines with $\gamma_{\xi_{n}}\left(  -\infty \right)  =\xi
_{0}=\gamma_{\xi}\left(  -\infty \right)  $ and $\gamma_{\xi_{n}}\left(
\infty \right)  =\xi_{n},$ $\gamma_{\xi}\left(  \infty \right)  =\xi.$ Pick a
parametrization for $\gamma_{\xi},$ for example, set $\gamma_{\xi}\left(
0\right)  $ to be a point of minimal distance from $\widetilde{s}$ and assume,
on the contrary, that $\xi \in I\left(  \rho \right)  .$

As above, let $\left[  x_{\eta},\widetilde{s}\right]  $ (resp. $\left[
x_{\rho},\widetilde{s}\right]  )$ be geodesic segments realizing the distance
$d_{\eta}$ (resp. $d_{\rho}$) of $\widetilde{s}$ from $\operatorname{Im}%
\gamma_{\eta}$ (resp. $\operatorname{Im}\gamma_{\rho}$) for some point
$x_{\eta}\in \operatorname{Im}\gamma_{\eta}$ (resp. $x_{\rho}\in
\operatorname{Im}\gamma_{\rho}$). Consider the neighborhood $\mathcal{N}$ of
geodesic lines around $\gamma_{\xi}$ determined by the positive number $C/2$
(cf. Lemma \ref{positive_distance}) and the compact set $K=\left[  -2\left(
d_{\eta}+d_{\rho}\right)  ,2\left(  d_{\eta}+d_{\rho}\right)  \right]  .$
Clearly, if $\gamma \in \mathcal{N}$ then $\gamma|_{\mathbb{R}\setminus K}$ does
not intersect the union of segments $\left[  x_{\eta},\widetilde{s}\right]
\cup \left[  x_{\rho},\widetilde{s}\right]  .$ The same holds for $\gamma_{n},$
$n$ large enough, thus, $\widetilde{s}\notin \widetilde{S}\left(  \left[
\gamma_{n}\left(  +\infty \right)  ,\rho \right]  \right)  $, equivalently,
$\xi_{n}\in I\left(  \rho \right)  $ a contradiction.
\end{proof}

We conclude this Section with the following proposition which indicates that
uniqueness of geodesic rays is a property which depends on the choice of base point.

\begin{proposition}
Let $\xi \in \partial \widetilde{S}$ be arbitrary. Then for some
point $x\in \widetilde{S}$ there exist at least two geodesic
rays $r_{1},$ $r_{2}$ such that $r_{1}\left(  0\right)  =r_{2}\left(
0\right)  =x$ and $r_{1}\left(  +\infty \right)  =r_{2}\left(
+\infty \right)  =\xi.$
\end{proposition}

\begin{proof}
Assume, on the contrary, that for every point $x\in \widetilde{S}$
there exists exactly one geodesic ray, denoted by $r_{x},$ with $r_{x}\left(
0\right)  =x$ and $r_{x}\left(  +\infty \right)  =\xi.$ Observe that for two
arbitrary distinct geodesic rays $r_{1},r_{2}$ with $r_{1}\left(
+\infty \right)  =r_{2}\left(  +\infty \right)  =\xi$ (then, by assumption,
$r_{1}\left(  0\right)  ,r_{2}\left(  0\right)  $ must be distinct) we have
that
\begin{equation}
\operatorname{Im}r_{1}\cap \operatorname{Im}r_{2}%
\mathrm{\mathrm{\ is\ either\ }a\mathrm{\ geodesic\ sub-ray~of\mathrm{~both~}%
or,\ }}\emptyset. \label{ray_intersection}%
\end{equation}
otherwise for a base point in the intersection we would have two distinct
geodesic rays corresponding to $\xi.$

Fix a point $\widetilde{s}_{0}$ where $s_{0}=p\left(  \widetilde{s}%
_{0}\right)  $ is a conical point with angle $\theta \left(  s_{0}\right)
<2\pi.$ Let $D\left(  \widetilde{s}_{0},\varepsilon \right)  $ be a closed disk
of radius $\varepsilon>0$ not containing any conical point except
$\widetilde{s}_{0}.$ The geodesic ray $r_{\widetilde{s}_{0}}$ intersects $\partial
\overline{D}\left(  \widetilde{s}_{0},\varepsilon \right)  $ at a single point
denoted $\widetilde{x}_{0}.$ We will reach a contradiction by defining a
continuous surjective map from $\partial D\left(  \widetilde{s}_{0}%
,\varepsilon \right)  \setminus \left \{  \widetilde{x}_{0}\right \}  $ to a space
$\left \{  +,-\right \}  $ consisting of two points.

Let $A$ be a positive number, see property (\ref{asymptoric_constant}) above,
such that for any $x\in D\left(  \widetilde{s}_{0},\varepsilon \right)  $ the
(unique) geodesic ray $r_{x}$ satisfies
\[
\forall t\in \left[  0,+\infty \right)  ,d\left(  r_{\widetilde{s}_{0}} \left(  t\right)
,r_{x}\left(  t\right)  \right)  <A.
\]
Observe that we may adjust $\varepsilon$ so that $\varepsilon<A.$  

Let $D\left(  r_{\widetilde{s}_{0}}\left(  3A\right)  ,A\right)  $ be the
closed disk of radius $A$ centered at $r_{\widetilde{s}_{0}}\left(  3A\right)
.$ Then, the set
\[
D\left(  r_{\widetilde{s}_{0}}\left(  3A\right)  ,A\right)  \setminus
\operatorname{Im}r_{\widetilde{s}_{0}}%
\]
consists of two connected components. Using the orientation
of $r_{\widetilde{s}_{0}}$ we may mark these components by saying that the
component to the right is the positive component and the one to the left the
negative, notation $D^{+}$ and $D^{-}$ respectively. In order to define a map
\[
R:\partial D\left(  \widetilde{s}_{0},\varepsilon \right)  \setminus \left \{
\widetilde{x}_{0}\right \}  \rightarrow \left \{  +,-\right \}
\]
we will distinguish 3 cases for each point $x\in \partial D\left(
\widetilde{s}_{0},\varepsilon \right)  \setminus \left \{  \widetilde{x}%
_{0}\right \}  :$\newline \underline{Case I:} $\operatorname{Im}r_{x}%
\cap \operatorname{Im}r_{\widetilde{s}_{0}}=\emptyset$\newline \underline{Case
II:} $\operatorname{Im}r_{x}\cap \operatorname{Im}r_{\widetilde{s}_{0}}%
\neq \emptyset$  and the unique time  $t_{x}\in [0,+\infty )$ so that \newline
$\operatorname{Im}%
r_{\widetilde{s}_{0}}|_{\left[  t_{x},+\infty \right]  }\subset
\operatorname{Im}r_{x},$
which exists by  (\ref{ray_intersection}), satisfies
$t_{x}>2A.$\newline \underline{Case III:} as in Case II with $t_{x}\leq2A .$

Observe that in Case III  $\operatorname{Im}r_{x}$ intersects neither $D^{+}$
nor $D^{-}.$ Clearly, in Cases I and II $\operatorname{Im}r_{x}$ intersects
at least one component $D^{+},$ $D^{-}.$  We claim that, in Cases I and II,
$\operatorname{Im}r_{x}$ cannot intersect both components $D^{+}$ and $D^{-}.$
To see this, let $\sigma:\left[  a,b\right]  \rightarrow \widetilde{S}$ be a
curve with the properties $\sigma \left(  a\right)  \in D^{+},\sigma \left(
b\right)  \in D^{-}$ and $\operatorname{Im}\sigma \cap \operatorname{Im}%
r_{\widetilde{s}_{0}}=\emptyset.$ By standard triangle inequality arguments it
follows that $lenght\left(  \sigma \right)  >2A.$ If  $\operatorname{Im}r_{x}$
intersected both $D^{+},D^{-}$ then, being a geodesic, it must intersect
$r_{\widetilde{s}_{0}}$ transversely, a contradiction according to the
assumptions in Case I and II. Thus, it follows (in Case I and II) that  either
$\operatorname{Im}r_{x}\cap$ $D^{+}\neq \emptyset$ or $\operatorname{Im}%
r_{x}\cap$ $D^{-}\neq \emptyset$ but not both. For $x\in \partial D\left(
\widetilde{s}_{0},\varepsilon \right)  \setminus \left \{  \widetilde{x}%
_{0}\right \}  $ whose geodesic ray $r_{x}$ falls into Case I or II we may now define

$R\left(  x\right)  :=+$ if $\operatorname{Im}r_{x}\cap$ $D^{+}%
\neq \emptyset$ and

$R\left(  x\right)  :=-$ if $\operatorname{Im}r_{x}\cap$ $D^{-}%
\neq \emptyset.$\newline Let now  $x\in \partial D\left(  \widetilde{s}%
_{0},\varepsilon \right)  \setminus \left \{  \widetilde{x}_{0}\right \}  $ so
that $r_{x}$ falls in to Case III. It is easy to see that $t_{x}>\varepsilon.$
For, if $0<t_{x}\leq \varepsilon$ then $r_{\widetilde{s}_{0}}\left(
t_{x}\right)  \in D\left(  \widetilde{s}_{0},\varepsilon \right)  $ which is
impossible because $r_{\widetilde{s}_{0}}\left(  t_{x}\right)  $ is a conical
point and $\varepsilon$ is chosen so that $\widetilde{s}_{0}$ is the unique
conical point in $D\left(  \widetilde{s}_{0},\varepsilon \right)  .$ If
$t_{x}=0,$ then the conical point $\widetilde{s}_{0}=r_{\widetilde{s}_{0}%
}\left(  0\right)  $ of angle $<2\pi$ lies on the geodesic ray $r_{x}$, a
contradiction by Lemma \ref{positive_distance}. Thus, $t_{x}>\varepsilon$ and
there exists $\delta>0$ sufficiently small so that  the disk $D\left(
r_{\widetilde{s}_{0}}\left(  t_{x}\right)  ,\delta \right)  $ does not contain
$\widetilde{x}_{0}=r_{\widetilde{s}_{0}}\left(  \varepsilon \right)  .$ This
disk $D\left(  r_{\widetilde{s}_{0}}\left(  t_{x}\right)  ,\delta \right)  $
can be used to define $R\left(  x\right)  $ as above: $\operatorname{Im}r_{x}$
intersects exactly one of the two oriented components of  $D\left(
r_{\widetilde{s}_{0}}\left(  t_{x}\right)  ,\delta \right)  \setminus
\operatorname{Im}r_{\widetilde{s}_{0}}$ and define $R\left(  x\right)  $ accordingly.

We show that $R$ is continuous. Let $\left \{  x_{n}\right \}  $ be a sequence
in $\partial D\left(  \widetilde{s}_{0},\varepsilon \right)  \setminus \left \{
\widetilde{x}_{0}\right \}  $ converging to a point $x.$ The sequence of
geodesic rays $\left \{  r_{x_{n}}\right \}  $ converges, up to a subsequence,
to a geodesic ray $q_{x}$ emanating from $x.$ Since $r_{x_{n}}\left(
+\infty \right)  =\xi$ for all $n,$ it follows that $q_{x}\left(
+\infty \right)  =\xi.$ By assumption of uniqueness of geodesic rays we have
$q_{x}=r_{x}.$ Thus, $r_{x_{n}}\rightarrow r_{x}$ uniformly on compact sets.
Without loss of generality we may  assume that $R\left(  x\right)  =+.$
\newline First assume that the geodesic ray $r_{x}$ falls into Case I or II,
that is, $\operatorname{Im}r_{x}\cap$ $D^{+}\neq \emptyset.$  Since $r_{x_{n}%
}\rightarrow r_{x}$ uniformly on compact sets it follows that there exists $N$
so that
\[
\forall n\geq N,\operatorname{Im}r_{x_{n}}\cap D^{+}\neq \emptyset
\]
which means that $R\left(  x_{n}\right)  =+,$ $\forall n\geq N.$ Case II is
treated similarly. This shows that $R$ is continuous. 

We show that $R$ is onto. We may choose a sequence $\left \{  x_{n}\right \}  $
$\subset \partial D\left(  \widetilde{s}_{0},\varepsilon \right)  \setminus
\left \{  \widetilde{x}_{0}\right \}  $ converging to $\widetilde{x}_{0}$ from
the right in the following sense: for all sufficiently small $\delta>0$ the
set $D\left(  \widetilde{x}_{0},\delta \right)  \setminus \operatorname{Im}%
r_{\widetilde{s}_{0}}$ consists of two connected components. We mark them as
right (positive) and left (negative) according to the positive direction of
$r_{\widetilde{s}_{0}}.$ We say that a sequence $\left \{  x_{n}\right \}  $
$\subset \partial D\left(  \widetilde{s}_{0},\varepsilon \right)  \setminus
\left \{  \widetilde{x}_{0}\right \}  $ converges to $\widetilde{x}_{0}$ from
the right if $x_{n}$ belongs to the right (positive) component of  $D\left(
\widetilde{x}_{0},\delta \right)  \setminus \operatorname{Im}r_{\widetilde
{s}_{0}}$ for all but finitely any $n.$ Clearly, for such a sequence 
$\left \{  x_{n}\right \}  ,$ the corresponding geodesic rays $r_{x_{n}}\rightarrow
r_{\widetilde{s}_{0}}|_{\left[  \varepsilon,+\infty \right)  }$ uniformly on
compact sets. Choose $\varepsilon_{1}<\varepsilon$ and set
\[
\mathcal{N}\left(  \varepsilon_{1}\right)  =\left \{  y\in \widetilde
{S}\bigm \vert \exists t\in \left[  \varepsilon,+\infty \right)  :d\left(
y,r_{\widetilde{s}_{0}}\left(  t\right)  \right)  <\varepsilon_{1}\right \}  .
\]
As above, $\mathcal{N}\left(  \varepsilon_{1}\right)  \setminus
\operatorname{Im}r_{\widetilde{s}_{0}}$ consists of two components
$\mathcal{N}^{+}$and $\mathcal{N}^{-}.$ Pick a sequence $x_{n}\rightarrow \widetilde
{x}_{0}$ from the right. Then, for all $n$ large enough, $x_{n}\in \mathcal{N}^{+}$
and by the uniform convergence of $r_{x_{n}},$ $\operatorname{Im}r_{x_{n}%
}\subset$ $\mathcal{N}^{+}.$ This implies that $R\left(  x_{n}\right)  =+$ for
all large enough $n.$ Similarly we show that $R$ attains the value $-.$
\end{proof}

\section{Density of closed geodesics}

We begin by showing that each class of geodesic rays (resp. lines) with the
same boundary point at infinity contains a leftmost and a rightmost geodesic
ray (resp. line) which bound a convex set containing the image of any other
geodesic ray (resp. line) in the same class.

\begin{proposition}
\label{left_right_rays}For every point $\xi \in \partial \widetilde{S},$ there
exist two geodesic rays $r_{L},r_{R}$ with $r_{L}\left(  \infty \right)
=\xi=r_{R}\left(  \infty \right)  $ and whose images bound a convex subset
$\widetilde{S}\left(  \xi \right)  $ of $\widetilde{S}$ with the property%
\[
\operatorname{Im}r\subset \widetilde{S}\left(  \xi \right)
\mathrm{\ for\ all\ geodesic\ rays\ }r\mathrm{\ with\ }r\left(  \infty \right)
=\xi.
\]

\end{proposition}

\begin{proof}
Observe that if the class of geodesic rays corresponding to $\xi$ is a
singleton, the above Proposition holds trivially with $r_{L}=r_{R}$ being the
unique geodesic ray with positive point $\xi.$

Let $A$ be the number posited in equation (\ref{asymptoric_constant}). Denote
by $C\left(  \widetilde{x}_{0},m\right)  $ (resp. $D\left(  \widetilde{x}%
_{0},m\right)  )$ the circle (resp. closed disc) of radius $m$ centered at
$\widetilde{x}_{0}.$ For each large enough $N\in \mathbb{N}$, the set%
\[
\xi \left(  N\right)  =\left \{  r\left(  N\right)  \bigm \vert r\in \xi \right \}
\]
is contained in an interval $I_{N,\xi}$ of diameter $A$ inside the circle
$C\left(  \widetilde{x}_{0},N\right)  .$ For large enough $N$ we may orient
$I_{N,\xi}$ and speak of its left and right endpoint.

We claim that $\xi \left(  N\right)  $ is a closed set. To see this let
$\left \{  y_{n}\right \}  $ be a sequence of points in $\xi \left(  N\right)  $
converging to $y\in I_{N,\xi}.$ By definition of $\xi \left(  N\right)  ,$ for
each $y_{n}$ there exists a geodesic ray $r_{n}\in \xi$ (not necessarily
unique) such that $r_{n}\left(  N\right)  =y_{n}.$ By passing to a
subsequence, if necessary, $\left \{  r_{n}\right \}  $ converges to a geodesic
ray $r$ and, clearly, $r\in \xi.$ As $y_{n}\rightarrow y,$ $y$ must belong to
$\operatorname{Im}r$ and, on the other hand, $y\in I_{N,\xi}\subset$ $C\left(
\widetilde{x}_{0},N\right)  .$ Thus, $y=r\left(  N\right)  $ which shows that
$\xi \left(  N\right)  $ is closed.

By compactness, the leftmost and rightmost points of $\xi \left(  N\right)  $
inside $I_{N,\xi},$ denoted by $y_{L}$ and $y_{R}$ respectively, exist. As the
number of conical points in $D\left(  \widetilde{x}_{0},N\right)  $ is finite,
we may choose (cf Lemma \ref{dyoBig}) geodesic segments $\sigma_{L,N}$ and
$\sigma_{R,N}$ with endpoints $\widetilde{x}_{0},y_{L}$ and $\widetilde{x}%
_{0},y_{R}$ respectively, satisfying the following property:

\begin{itemize}
\item the convex subset of $D\left(  \widetilde{x}_{0},N\right)  $ bounded by
the union%
\begin{equation}
\sigma_{L,N}\cup \left[  y_{L},y_{R}\right]  \cup \sigma_{R,N}
\label{chain_increasin0}%
\end{equation}
where $\left[  y_{L},y_{R}\right]  $ indicates the subinterval of $C\left(
\widetilde{x}_{0},N\right)  $ containing $\xi \left(  N\right)  ,$ contains all
geodesic segments $r|_{\left[  0,N\right]  }$ for all $r\in \xi.$
\end{itemize}

The segment $\sigma_{L,N}$ (and similarly for $\sigma_{R,N})$ can be obtained
by starting with a geodesic segment $\sigma_{L,N}^{\prime}$ with endpoints
$\widetilde{x}_{0},y_{L}$ and then if a geodesic ray intersects the segment
$\sigma_{L,N}^{\prime}$, it must do so at pairs of (conical) points
(otherwise, the property of $y_{L}$ being leftmost would be violated). As the
intersection points are conical points, they are finitely many pairs of
intersection points so we may replace (see Lemma \ref{dyoBig}) finitely many
parts of the segment $\sigma_{L,N}^{\prime}$ to obtain $\sigma_{L,N}.$

The sequences $\left \{  \sigma_{L,N}\right \}  _{N\in \mathbb{N}},$ $\left \{
\sigma_{R,N}\right \}  _{N\in \mathbb{N}}$ converge to geodesic rays
$r_{L},r_{R}\in \xi$ respectively. The required property in the statement of
the Proposition for the convex set $\widetilde{S}\left(  \xi \right)  $ bounded
by $\operatorname{Im}r_{L},$ $\operatorname{Im}r_{R},$ now follows: for, if
$r\in \xi$ with $\operatorname{Im}r\nsubseteqq \widetilde{S}\left(  \xi \right)
$ then, for some $M>0,$ $r\left(  M\right)  \notin$ $\widetilde{S}\left(
\xi \right)  .$ Assume that the distance $d\left(  r\left(  M\right)
,\widetilde{S}\left(  \xi \right)  \right)  =C_{0}>0$ of $r\left(  M\right)  $
from $\widetilde{S}\left(  \xi \right)  $ is realized by a point on
$\operatorname{Im}r_{R}.$ Then, for a compact set $K\supset \left[  0,M\right]
$ and the positive number $C_{0}/2$ there exist $N_{0}$ so that
\[
d\left(  r_{R}\left(  t\right)  ,\sigma_{R,N}\left(  t\right)  \right)
<C_{0}/2\mathrm{\ for\ all\ }t\in K\mathrm{\ and\ for\ all\ }N>N_{0}.
\]
We may assume that $N_0$ satisfies $N_0 > [M]+1 .$ It follows that $r\left(  M\right)  $ 
does not belong to the convex subset of
$D\left(  \widetilde{x}_{0},N_0\right)  $ bounded by the union
\[
\sigma_{L,N_0}\cup \left[  y_{L},y_{R}\right]  \cup \sigma_{R,N_0}%
\]
contradicting (\ref{chain_increasin0}).
\end{proof}

In view of the above Proposition we introduce the following

\noindent \textbf{Terminology: }For each $\xi \in \partial \widetilde{S},$ the
geodesic rays posited in the above proposition will be called leftmost and
rightmost geodesic rays in the class of $\xi$ and will be denoted by
$r_{L,\xi}$ and $r_{R,\xi}$ respectively.

\begin{proposition}
\label{left_right_lines}For every pair of points $\eta,\xi \in \partial
\widetilde{S}$ with $\eta \neq \xi,$ there exist two geodesic lines $\gamma
_{L},\gamma_{R}\in \left(  \eta,\xi \right)  ,$ that is, $\gamma_{L}\left(
-\infty \right)  =\eta=\gamma_{R}\left(  -\infty \right)  $ and $\gamma
_{L}\left(  \infty \right)  =\xi=\gamma_{R}\left(  \infty \right)  ,$ whose
images bound a convex subset $\widetilde{S}\left(  \eta,\xi \right)  $ of
$\widetilde{S}$ with the property%
\[
\operatorname{Im}\gamma \subset \widetilde{S}\left(  \eta,\xi \right)
\mathrm{\ for\ all\ geodesic\ lines\ }\gamma \in \left(  \eta,\xi \right)  .
\]

\end{proposition}

\begin{proof}
The line of proof is similar to the previous proposition, however, we include
it here because certain modifications are needed.

We may assume that there exist at least two geodesics in the class of $\left(
\eta,\xi \right)  ,$ otherwise the statement is trivial. Moreover, each
$\gamma \in \left(  \eta,\xi \right)  $ is considered oriented with positive the
direction from $\eta$ to $\xi$ and then the left and right component of
$\partial \widetilde{S}\setminus \left \{  \eta,\xi \right \}  $ is determined.
Pick a base point $\widetilde{x}_{0}$ on the image of an arbitrary $\gamma
_{0}\in \left(  \eta,\xi \right)  $ and set $\gamma_{0}\left(  0\right)
=\widetilde{x}_{0}.$ For large enough $N\in \mathbb{N},$ we may find intervals
\[%
\begin{array}
[c]{c}%
I_{N}^{+}\left(  \eta,\xi \right)  :=\left[  \gamma_{0}\left(  N\right)
-A,\gamma_{0}\left(  N\right)  +A\right]  \subset C\left(  \widetilde{x}%
_{0},N\right) \\
I_{N}^{-}\left(  \eta,\xi \right)  :=\left[  \gamma_{0}\left(  -N\right)
-A,\gamma_{0}\left(  -N\right)  +A\right]  \subset C\left(  \widetilde{x}%
_{0},N\right)
\end{array}
\]
where $A$ is the constant posited in (\ref{asymptoric_constant}), with the
property%
\[
\mathrm{for\ all\ }\gamma \in \left(  \eta,\xi \right)  ,\operatorname{Im}%
\gamma \cap C\left(  \widetilde{x}_{0},N\right)  \subset I_{N}^{+}\left(
\eta,\xi \right)  \cup I_{N}^{+}\left(  \eta,\xi \right)  .
\]
For each $\gamma$ in $\left(  \eta,\xi \right)  ,$ the intersection
$\operatorname{Im}\gamma \cap I_{N}^{+}\left(  \eta,\xi \right)  $ (resp.
$I_{N}^{-}\left(  \eta,\xi \right)  $) is not necessarily a singleton. However,
there exist unique numbers $t_{\gamma,N}^{-},t_{\gamma,N}^{+}\in \mathbb{R}$
such that
\[
\gamma \left(  t_{\gamma,N}^{-}\right)  \in I_{N}^{-}\left(  \eta,\xi \right)
,\gamma \left(  t_{\gamma,N}^{+}\right)  \in I_{N}^{+}\left(  \eta,\xi \right)
\]
and $\left \vert t_{\gamma,N}^{-}-t_{\gamma,N}^{+}\right \vert $ is minimal with
respect to the above inclusions. Equivalently, $\gamma|_{\left(  t_{\gamma
,N}^{-},t_{\gamma,N}^{+}\right)  }\cap C\left(  \widetilde{x}_{0},N\right)
=\emptyset.$ Set
\[
\xi^{+}\left(  N\right)  =\left \{  \gamma \left(  t_{\gamma,N}^{+}\right)
\bigm \vert \gamma \in \left(  \eta,\xi \right)  \right \}
\]
and similarly for $\xi^{-}\left(  N\right)  .$ We claim that both sets
$\xi^{+}\left(  N\right)  ,\xi^{-}\left(  N\right)  $ are closed. To see this
let $\left \{  y_{n}\right \}  $ be a sequence of points in $\xi^{+}\left(
N\right)  $ converging to $y\in I_{N}^{+}\left(  \eta,\xi \right)  .$ By
definition of $\xi^{+}\left(  N\right)  ,$ for each $y_{n}$ there exists a
geodesic $\gamma_{n}\in \left(  \eta,\xi \right)  $ (not necessarily unique)
such that $\gamma_{n}\left(  t_{\gamma_{n},N}^{+}\right)  =y_{n}.$ By passing
to a subsequence, if necessary, $\left \{  \gamma_{n}\right \}  $ converges to a
geodesic $\gamma^{\prime}.$ Clearly, $\gamma^{\prime}\in \left(  \eta
,\xi \right)  $ and $y$ must belong to $\operatorname{Im}\gamma^{\prime}\cap
I_{N}^{+}\left(  \eta,\xi \right)  .$ In order to complete the proof that
$\xi^{+}\left(  N\right)  $ is closed we need to show that $y=\gamma^{\prime
}\left(  t_{\gamma^{\prime},N}^{+}\right)  .$

We may assume that $N$ is large enough so that
\[
\left \vert t_{\gamma_{n},N}^{-}-t_{\gamma_{n},N}^{+}\right \vert >2A
\]
for all $n.$ Observe that for every $\varepsilon>0$ which belongs to the
interval $\left(  0,\left \vert t_{\gamma_{n},N}^{-}-t_{\gamma_{n},N}%
^{+}\right \vert \right)  $ for all $n,$ the sequence $\left \{  \gamma
_{n}\left(  t_{\gamma_{n},N}^{+}-\varepsilon \right)  \right \}  _{n\in
\mathbb{N}}$ converges to the interior of the disk $D\left(  \widetilde{x}%
_{0},N\right)  .$ Therefore, all points on $\operatorname{Im}\gamma^{\prime
}|_{\left(  \eta,y\right]  }$ of distance $\varepsilon<2A$ from $y$ belong to
the interior of the disk $D\left(  \widetilde{x}_{0},N\right)  .$ It follows
that $y=\gamma^{\prime}\left(  t_{\gamma^{\prime},N}^{+}\right)  $ which shows
that $\xi^{+}\left(  N\right)  $ is closed. Similarly we show that $\xi
^{-}\left(  N\right)  $ is closed.

Denote by $y_{L}^{+}$ (resp. $y_{R}^{+})$ be the leftmost (resp. rightmost)
point of $\xi^{+}\left(  N\right)  $ in $I_{N}^{+}\left(  \eta,\xi \right)  $
and $y_{L}^{-}$ (resp. $y_{R}^{-})$ be the leftmost (resp. rightmost) point of
$\xi^{-}\left(  N\right)  $ in $I_{N}^{-}\left(  \eta,\xi \right)  $ all for
which exist by compactness. We may construct a rightmost geodesic segment
$\sigma_{R,N}=\left[  y_{R}^{-},y_{R}^{+}\right]  $ in $D\left(  \widetilde
{x}_{0},N\right)  $ with endpoints $y_{R}^{-},y_{R}^{+}$ and a leftmost
geodesic segment $\sigma_{L,N}=\left[  y_{L}^{-},y_{L}^{+}\right]  $ in
$D\left(  \widetilde{x}_{0},N\right)  $ with endpoints $y_{L}^{-},y_{L}^{+}$
so that the following property holds:

\begin{itemize}
\item the convex subset of $D\left(  \widetilde{x}_{0},N\right)  $ bounded by
the union%
\[
\sigma_{L,N}\cup \left[  y_{L}^{+},y_{R}^{+}\right]  \cup \sigma_{R,N}%
\cup \left[  y_{R}^{-},y_{R}^{+}\right]
\]
where $\left[  y_{L}^{+},y_{R}^{+}\right]  $ (resp. $\left[  y_{L}^{-}%
,y_{R}^{-}\right]  $) indicates the subinterval of $C\left(  \widetilde{x}%
_{0},N\right)  $ containing $\xi^{+}\left(  N\right)  $ (resp. $\xi^{-}\left(
N\right)  $), contains all segments $\operatorname{Im}\gamma \cap D\left(
\widetilde{x}_{0},N\right)  $ for all $\gamma \in \left(  \eta,\xi \right)  .$
\end{itemize}

The segment $\sigma_{L,N}$ (and similarly for $\sigma_{R,N})$ can be obtained
by starting with a geodesic segment $\sigma_{L,N}^{\prime}$ with endpoints
$y_{L}^{-},y_{L}^{+}$ and then if a geodesic line intersects the segment
$\sigma_{L,N}^{\prime}$, it must do so at pairs of (conical) points
(otherwise, the property of $y_{L}^{+},y_{L}^{-}$ being leftmost would be
violated). As the intersection points are conical points, they are finitely
many so we may replace (see Lemma \ref{dyoBig}) finitely many parts of the
segment $\sigma_{L,N}^{\prime}$ to obtain $\sigma_{L,N}.$

As in the proof of the previous proposition we obtain the desired geodesic
lines as limits of the sequences $\left \{  \sigma_{L,N}\right \}
_{N\in \mathbb{N}}$ and $\left \{  \sigma_{R,N}\right \}  _{N\in \mathbb{N}}.$
\end{proof}

\noindent \textbf{Terminology: }For each \bigskip$\left(  \eta,\xi \right)
\in \partial^{2}\widetilde{S},$ the geodesic lines posited in the above
proposition will be called leftmost and rightmost geodesic lines in the class
of $\left(  \eta,\xi \right)  $ and will be denoted by $\gamma_{L,\left(
\eta,\xi \right)  }$ and $\gamma_{R,\left(  \eta,\xi \right)  }$ respectively.

\begin{theorem}
\label{mmain}Closed geodesics are dense in $GS$ in the following sense: for
each pair \bigskip$\left(  \eta,\xi \right)  \in \partial^{2}\widetilde{S}$
there exists a sequence of geodesics $\left \{  c_{n}\right \}  $ such that
$c_{n}\rightarrow \gamma_{L,\left(  \eta,\xi \right)  }$ in the usual uniform
sense on compact sets and $p(c_{n})$ is a closed geodesic in $S$ for all $n.$
Similarly for $\gamma_{R,\left(  \eta,\xi \right)  }.$
\end{theorem}

\begin{proof}
For arbitrary $\left(  \eta,\xi \right)  \in \partial^{2}\widetilde{S},$ we
orient as positive the direction from $\eta$ to $\xi$ and name left and right
the components of $\partial \widetilde{S}\setminus \left \{  \eta,\xi \right \}  .$
We may choose, by Proposition \ref{dense3}, a sequence $\left \{  \left(
\phi_{n}\left(  -\infty \right)  ,\phi_{n}\left(  +\infty \right)  \right)
\right \}  _{n\in \mathbb{N}}$ where each $\phi_{n}$ is a (hyperbolic) element
of $\pi_{1}\left(  S\right)  $ such that $\phi_{n}\left(  -\infty \right)
\rightarrow \eta,$ $\phi_{n}\left(  \infty \right)  \rightarrow \xi$ with the
additional property that for all $n$ both $\phi_{n}\left(  -\infty \right)  ,$
$\phi_{n}\left(  \infty \right)  $ belong to the same (say, right) component of
$\partial \widetilde{S}\setminus \left \{  \eta,\xi \right \}  .$ In particular we
have that $\phi_{n}\left(  -\infty \right)  \neq \eta$ and $\phi_{n}\left(
\infty \right)  \neq \xi.$

We claim that for each $n\in \mathbb{N},$ there exists a geodesic $c_{n}%
^{\prime \prime}\in \left(  \phi_{n}\left(  -\infty \right)  ,\phi_{n}\left(
+\infty \right)  \right)  ,$ that is, $c_{n}^{\prime \prime}\left(
-\infty \right)  =\phi_{n}\left(  -\infty \right)  $ and $c_{n}^{\prime \prime
}\left(  \infty \right)  =\phi_{n}\left(  +\infty \right)  $ whose projection to
$S\ $is closed. To see this, pick arbitrary $y\in \widetilde{S}$ and consider
the geodesic segment $\left[  y,\phi_{n}\left(  y\right)  \right]  $ which,
clearly, projects to a closed curve, say $\sigma_{n}$ in $S.$ There exists a
length minimizing closed curve in the (free) homotopy class of $\sigma_{n}$
(see \cite[Ch. 1, Remark 1.13(b)]{[GLP]}). By choosing an appropriate lift to $\widetilde{S}$ of this
length minimizing closed curve we obtain a geodesic line $c_{n}^{\prime \prime
}$ such that the set
\[
\left \{  \phi_{n}^{i}\left(  y\right)  \bigm \vert i\in \mathbb{Z}\right \}
\]
is at bounded distance from $\operatorname{Im}$ $c_{n}^{\prime \prime}.$ Thus,
$c_{n}^{\prime \prime}\left(  -\infty \right)  =\phi_{n}\left(  -\infty \right)
$ and $c_{n}^{\prime \prime}\left(  \infty \right)  =\phi_{n}\left(
+\infty \right)  $ as desired.

Since by construction the projection of $c_{n}^{\prime \prime}$ to $S$ is a
closed curve, we can speak of the period of $c_{n}^{\prime \prime}.$ Let
$\gamma_{R,\left(  \eta,\xi \right)  }\in \left(  \eta,\xi \right)  $ be the
rightmost geodesic posited in Proposition \ref{left_right_lines}. As
$c_{n}^{\prime \prime}\left(  -\infty \right)  \neq \eta$ and $c_{n}%
^{\prime \prime}\left(  \infty \right)  \neq \xi,$ the intersection
\[
\operatorname{Im}c_{n}^{\prime \prime}\cap \operatorname{Im}\gamma_{R,\left(
\eta,\xi \right)  }%
\]
has finitely many components. For each $n\in \mathbb{N},$ consider the geodesic
line $c_{n}^{\prime}$ having the same image as $c_{n}^{\prime \prime}$ and its
period is a multiple of the period of $c_{n}^{\prime \prime}$ so that
\[
\operatorname{Im}c_{n}^{\prime}\cap \operatorname{Im}\gamma_{R,\left(  \eta
,\xi \right)  }%
\]
is contained in a single period of $c_{n}^{\prime}.$ We may alter
$c_{n}^{\prime}$ in its (enlarged) period so that it does not intersect the
interior of the convex subset $\widetilde{S}\left(  \eta,\xi \right)  $ of
$\widetilde{S}$ bounded by $\operatorname{Im}\gamma_{L,\left(  \eta
,\xi \right)  }$ and $\operatorname{Im}\gamma_{R,\left(  \eta,\xi \right)  }.$
For such an alteration we only need to modify $c_{n}^{\prime}$ in
subintervals, say $\left[  z,w\right]  ,$ of its image contained in
$\widetilde{S}\left(  \eta,\xi \right)  .$ Namely, we have to replace
$c_{n}^{\prime}|_{\left[  z,w\right]  }$ by $\gamma_{R,\left(  \eta
,\xi \right)  }|_{\left[  z,w\right]  }.$ Then, by repeating this alteration we
obtain a geodesic line, denoted by $c_{n},$ whose projection to $S\ is$ a
closed geodesic. Clearly by construction

$c_{n}\left(  -\infty \right)  =c_{n}^{\prime \prime}\left(  -\infty \right)
=\phi_{n}\left(  -\infty \right)  $, $c_{n}\left(  \infty \right)
=c_{n}^{\prime \prime}\left(  +\infty \right)  =\phi_{n}\left(  +\infty \right)
$ and\newline as $\operatorname{Im}c_{n}\cap \widetilde{S}\left(  \eta
,\xi \right)  \subset \operatorname{Im}\gamma_{R,\left(  \eta,\xi \right)  },$ it
follows that $c_{n}\rightarrow \gamma_{L,\left(  \eta,\xi \right)  }$ uniformly
on compact sets.
\end{proof}


\begin{thebibliography}{9}                                                                                                %


\bibitem {[CPT]}Ch. Charitos, I. Papadoperakis, G. Tsapogas, \textit{The
geometry of Euclidean surfaces with conical singularities, }http://arxiv.org/abs/1306.1759

\bibitem {[CDP]}M. Coornaert, T. Delzant, A. Papadopoulos,
\textit{G\'{e}ometrie et th\'{e}orie des groupes,} Lecture Notes in
Mathematics, 1441, Sringer-Verlag, 1980.

\bibitem {[Coo]}M. Coornaert, \textit{Sur les gropes proprement discontinus
d'isometries des espaces hyperboliques au sens de Gromov}, Th\`{e}se de
U.L.P., Publication de I.R.M.A., 1990.

\bibitem {Gr}M. Gromov, \textit{Hyperbolic groups} in \textquotedblleft Essays
in group theory\textquotedblleft, M.S.R.I. Publ. 8 (Springer 1987), 75-263.

\bibitem {[GLP]}M. Gromov, J. Lafontaine, P. Pansu, \textit{Structures
metriques pour les vari\'{e}t\'{e}s Riemanniennes}, Fernand Nathan, Paris, 1981.
\end{thebibliography}
\end{document}